\documentclass[12pt, reqno]{amsart}
\usepackage{amssymb,amsthm,amsfonts,amsmath}

\usepackage{hyperref}
\usepackage{mathrsfs}

\newcommand{\Phit}{\widetilde{\Phi}}

\newcommand{\giv}{\,|\,}

\newcommand{\F}{\mathbb F}
\newcommand{\Ga}{\Gamma}
\newcommand{\de}{\delta}
\newcommand{\Gr}{\operatorname{Gr}}
\newcommand{\G}{\mathscr G}
\newcommand{\ka}{\varkappa}
\newcommand{\epsi}{\varepsilon}
\newcommand{\De}{\Delta}
\newcommand{\prob}{\mathbb P}

\newcommand{\V}{\mathcal{V}}
\newcommand{\tV}{\widetilde{\V}}
\newcommand{\tv}{\tilde v}

\newtheorem{theorem}{Theorem}[section]
\newtheorem{proposition}[theorem] {Proposition}

\newtheorem{corollary}[theorem]{Corollary}
\newtheorem{lemma}[theorem]{ Lemma}

\theoremstyle{definition}
\newtheorem{definition}[theorem]{Definition}
\newtheorem{remark}[theorem]{Remark}

\begin{document}
\title{A $q$-analogue of de Finetti's theorem}
\author{Alexander Gnedin}
\address{Department of Mathematics, Utrecht University, the Netherlands}
\email{A.V.Gnedin@uu.nl}

\author{Grigori Olshanski}
\address{Institute for Information Transmission Problems, Moscow, and
Independent University of Moscow, Russia}
\email{olsh2007@gmail.com}

\thanks{G.~O. was supported by a grant from the Utrecht
University, by the RFBR grant 08-01-00110, and by the project SFB 701
(Bielefeld University).}

\begin{abstract}
\noindent A $q$-analogue of de Finetti's theorem is obtained in terms of a
boundary problem for the $q$-Pascal graph. For $q$ a power of prime this leads
to a characterisation of random spaces over the Galois field $\F_q$ that are
invariant under the natural action of the infinite group of invertible matrices
with coefficients from  $\F_q$.
\end{abstract}

\maketitle

\section{Introduction}\label{1}
The infinite symmetric group  ${\mathfrak S}_\infty$ consists of bijections
$\{1,2,\dots\}\to\{1,2,\dots\}$ which move only finitely many integers. The
group ${\mathfrak S}_\infty$ acts on the product space $\{0,1\}^\infty$ by
permutations of the coordinates. A random element of this space, that is a
random infinite binary sequence, is called {\it exchangeable} if its
probability law is invariant under the action of ${\mathfrak S}_\infty$. De
Finetti's theorem asserts that every exchangeable sequence can be generated in
a unique way by the following two-step procedure: first choose at random the
value of parameter $p$ from some probability distribution on the unit interval
$[0,1]$, then run an infinite Bernoulli process with probability $p$ for $1$'s.

One approach to this classical result, as presented in Feller \cite[Ch. VII,
\S4]{Feller}, is based on the following exciting connection with the Hausdorff
moment problem. By exchangeability, the law of a  random  infinite binary
sequence is determined by the array  $(v_{n,k})$, where $v_{n,k}$ equals the
probability of every initial sequence of length $n$ with $k$ $1$'s. The rule of
addition of probabilities yields the backward recursion
\begin{equation}\label{dP}
v_{n,k}=v_{n+1,k}+v_{n+1,k+1}, ~~~0\leq k\leq n, ~n=0,1,\ldots,
\end{equation}
which readily implies that the array can be derived by iterated differencing of
the  sequence $(v_{n,0})_{n=0,1,\dots}$. Specifically, setting
\begin{equation}\label{uv}
u^{(k)}_l=v_{l+k,k}, \qquad l=0,1,\dots, \quad k=0,1,\dots,
\end{equation}
and denoting by $\de$ the  difference operator acting on sequences
$u=(u_l)_{l=0,1,\dots}$ as
$$
(\de u)_l=u_{l}-u_{l+1},
$$
the recursion \eqref{dP} can be written as
\begin{equation}\label{de}
u^{(k)}= \de u^{(k-1)}, \qquad k=1,2,\dots\,.
\end{equation}
Since $v_{n,k}\geq 0$, the sequence $u^{(0)}$ must be completely monotone, that
is, componentwise
$$
\underbrace{\de\circ\dots\circ\de}_k u^{(0)}\ge0, \qquad k=0,1,\dots,
$$
but then Hausdorff's theorem implies that there exists a representation
\begin{equation}\label{Haus}
v_{n,k} =u^{(k)}_{n-k}=\int_{[0,1]} p^k(1-p)^{n-k}\mu({\rm d}p)
\end{equation}
with uniquely determined  probability measure $\mu$. De Finetti's theorem
follows since $v_{n,k}=p^k(1-p)^{n-k}$ for the Bernoulli process with parameter
$p$. See \cite{Aldous} for other proofs and extensive survey of generalisations
of this result.

The present note is devoted to variations on the $q$-analogue of de Finetti's
theorem,  which was briefly outlined in Kerov \cite{Kerov} within the framework
of the boundary problem for generalised Stirling triangles. The boundary
problem for other weighted versions of the Pascal triangle was studied in
\cite{Euler}, \cite{Gibbs},
and for more general graded graphs in
\cite{Zigzag}, \cite{Kerov}, \cite{KOO}.

\begin{definition}\label{q-exch}
Given $q>0$, let us say that a random binary sequence
$\epsi=(\epsi_1,\epsi_2,\dots)\in\{0,1\}^\infty$ is {\it $q$-exchangeable\/} if
its probability law $\prob$ is $\mathfrak S_\infty$-quasiinvariant with a
specific cocycle, which is uniquely determined by the following condition:
Denoting by $\prob(\epsi_1,\dots,\epsi_n)$ the probability of an initial
sequence $(\epsi_1,\dots,\epsi_n)$, we have for any $i=1,\dots,n-1$
$$
\prob(\epsi_1,\dots,\epsi_{i-1},\epsi_{i+1},\epsi_i,
\epsi_{i+2},\dots,\epsi_n)=
q^{\epsi_i-\epsi_{i+1}}\prob(\epsi_1,\dots,\epsi_n).
$$
In words: under an elementary transposition of the form $(\ldots,1,0,\ldots)
\to (\ldots,0,1,\ldots)$, probability is multiplied by $q$.
\end{definition}

\begin{theorem}\label{q-de-F}
Assume $0<q<1$. There is a bijective correspondence $\prob\leftrightarrow\mu$
between the probability laws\/ $\prob$ of infinite $q$-exchangeable binary
sequences and the probability measures $\mu$ on the closed countable set
$$
\Delta_q:=\{1,q,q^2,\ldots\}\cup\{0\}\subset[0,1].
$$
\end{theorem}

More precisely, a $q$-exchangeable sequence can be generated in a unique way by
first choosing at random a point $x\in\De_q$ distributed according to $\mu$ and
then running a certain $q$-analogue of the Bernoulli process indexed by $x$.
Each law $\prob$ is uniquely determined by the infinite triangular
array
\begin{equation}\label{qvnk}
v_{n,k}:=\prob(\,\underbrace{1,\dots,1}_k,\underbrace{0,\dots,0}_{n-k}\,),
\qquad 0\le k\le n<\infty,
\end{equation}
which in turn is given by a $q$-version of formula \eqref{Haus}, with $\De$ being
replaced by $\De_q$ (Theorem \ref{main}).
A similar result with switching the roles of $0$'s and $1$'s and replacing $q$
by $q^{-1}$ also holds for $q>1$.

The rest of the paper is organized as follows. In Section \ref{2} we introduce
the $q$-Pascal graph and formulate the $q$-exchangeability in terms of certain
Markov chains on this graph. In Section \ref{3} we find a characteristic
recursion for the numbers \eqref{qvnk}, which is a $q$-deformation of
\eqref{dP}, and we prove the main result, equivalent to Theorem \ref{q-de-F},
using the method of \cite{KOO}. In Section \ref{4} we discuss three examples:
two $q$-analogues of the Bernoulli process and a $q$-analogue of P\'olya's urn
process. Finally, in Section \ref{5}, for $q$ a power of a prime number, we
provide an interpretation of the theorem in terms of random subspaces in an
infinite-dimensional vector space over $\F_q$.

\section{The $q$-Pascal graph}\label{2}

For $q>0$, the $q$-Pascal graph is a {\it weighted\/} directed graph $\Ga(q)$
on the infinite vertex set
$$
\Gamma =\{(l,k): l,k=0,1,\ldots\}.
$$
Each vertex $(l,k)$ has two weighted outgoing edges $(l,k)\to(l+1,k)$ and
$(l,k)\to(l,k+1)$ with weights $1$ and $q^l$, respectively. The vertex set is
divided into levels $\Gamma_n=\{(l,k): l+k=n\}$, so $\Gamma=\cup_{n\geq 0}
\Gamma_n$ with $\Gamma_0$ consisting of the sole root vertex $(0,0)$. For a
path in $\Gamma$ connecting two vertices $(l,k)\in \Gamma_{l+k}$ and
$(\lambda,\varkappa)\in \Gamma_{\lambda+\varkappa}$ we define the weight to be
the product of weights of edges along the path. For instance, the weight of
$(2,3)\to(2,4)\to(3,4)\to(3,5)$ is $q^5=q^2\cdot 1\cdot q^3$. Clearly, such a
path exists if and only if $\lambda\ge l$, $\varkappa\ge k$.

We shall consider certain transient Markov chains $S=(S_n),$
with state-space $\Ga$, which start at the root $(0,0)$ and move along the
directed edges, so that $S_n\in\Ga_n$ for every $n=0,1,\dots$.
Thus, a trajectory of $S$
is an infinite directed path in $\Ga$ started at the root.

\begin{definition}\label{centr}
Adopting the terminology introduced by Vershik and Kerov (see \cite{Kerov}), we
say that a Markov chain $S$ on $\Ga(q)$ is {\it central\/} if the following
condition is satisfied for each vertex $(n-k,k)\in\Ga_n$ visited by $S$ with
positive probability: given $S_n=(n-k,k)$, the conditional probability that $S$
follows each particular path connecting $(0,0)$ and
 $(n-k,k)$ is proportional to the weight of the path.
\end{definition}

\begin{remark}\label{M-pr}
If we only require the centrality condition to hold for all $(l,k)\in
\Gamma_\nu$ for fixed $\nu$, then we have it satisfied also for all $(l,k)$
with $l+k\leq\nu$. From this it is easy to see that the centrality condition
{\it implies} the Markov property of $S$ in reversed time $n=\ldots,1,0$, hence
also implies the Markov property in forward time $n=0,1,\ldots$.
\end{remark}

In the special case $q=1$ Definition \ref{centr} means  that in the Pascal
graph $\Gamma(1)$   all paths with common endpoints are equally likely.

Recall a  bijection between the infinite binary sequences
$(\epsi_1,\epsi_2,\dots)$ and infinite directed paths in $\Gamma$ started at
the root $(0,0)$. Specifically, given a path, the $n$th digit $\epsi_n$ is
given the value 0 or 1 depending on whether $l$ or $k$ coordinate is increased
by 1. Indentifying a path with a sequence  $(n-K_n,K_n)$ (where $0\leq K_n\leq
n$), the correspondence can be written as
$$
K_n=\sum_{j=1}^n \epsi_j\,,~~~~\epsi_n=K_n-K_{n-1},~~~~n=1,2,\ldots.
$$

\begin{proposition}\label{prop2A}
By virtue of the bijection between $\{0,1\}^\infty$ and the paths in~ $\Gamma$,
each~ $q$-exchangeable sequence corresponds to a central Markov chain on
$\Gamma(q)$, and vice versa.
\end{proposition}

\begin{proof}
This follows readily from Remark \ref{M-pr}, Definitions \ref{q-exch} and
\ref{centr} and the structure of $\Ga(q)$.
\end{proof}

We shall use the standard notation
$$
[n]:=1+q+\ldots+q^{n-1},~~~[n]!:=[1]\cdot[2]\cdots[n],~~~ \left[\!\!
\begin{array}{c} n\\ k \end{array}\!\!\right] :=\frac{[n]!}{[k]![n-k]!}
$$
for $q$-integers, $q$-factorials and $q$-binomial coefficients, respectively,
with the usual convention that
$
\left[\!\! \begin{array}{c} n\\ k \end{array}\!\!\right]=0
$
for $n<0$ or $k<0$. Furthermore, we set
$$(x,q)_k:=\prod_{i=0}^{k-1} (1-xq^{i})\,,~~~~1\leq k\leq\infty,$$
with the infinite product $(k=\infty)$ considered for $0<q<1$.

The following lemma justifies the name of the graph by relating it to the
$q$-Pascal  triangle of $q$-binomial coefficients.

\begin{lemma}\label{dim}
The sum of weights of all directed paths from the root $(0,0)$ to a vertex
$(n-k,k)$, denoted $d_{n,k}$, is given by
\begin{equation}\label{dim}
d_{n,k}=\left[\!\! \begin{array}{c} n\\ k \end{array}\!\!\right].
\end{equation}
More generally, $d_{n,k}^{\nu,\varkappa}$, the sum of weights of all paths
connecting two vertices $(n-k,k)$ and $(\nu-\varkappa,\varkappa)$ in $\Gamma$
is given by
$$
d_{n,k}^{\nu,\varkappa}= q^{(\varkappa-k) (n-k)}
\left[\!\! \begin{array}{c} \nu-n\\ \varkappa-k \end{array}\!\!\right].
$$

\end{lemma}
\begin{proof} Note that any   path from $(0,0)$ to $(n-k,k)$ has the second
component incrementing by $1$ on some $k$ edges $(l_i,i-1)\to(l_i,i)$, where
$i=1,2,\dots,k$ and $0\le l_1\le\dots\le l_k\le n-k$, thus the sum of weights
is equal to
\begin{equation}\label{sow}
d_{n,k}= \sum_{0\le l_1\le \dots \le l_k\le n-k} q^{l_1+\dots+l_k}.
\end{equation}
This array
satisfies the recursion
\begin{equation}\label{rec1}
d_{n,k}=q^{n-k}d_{n-1,k-1}+d_{n-1,k}\,,~~~~~0<k<n
\end{equation}
with the boundary conditions $d_{n,0}=d_{n,n}=1$. On the other hand, it is well
known that the array of $q$-binomial coefficients also satisfies  this
recursion \cite{Kac}, hence by the uniqueness $d_{n,k}$ is the $q$-binomial
coefficient. In the like way the sum of weights of paths from $(n-k,k)$ to
$(\nu-\varkappa,\varkappa)$ is
$$
d_{n,k}^{\nu,\varkappa}=\sum_{n-k\le l_1\le \dots \le l_{k'}\le \nu-\varkappa}
q^{l_1+\dots+l_{k'}}, \quad k':=\varkappa-k.
$$
Comparing with (\ref{sow}) we see that this is equal to $q^{(n-k)k'} \left[\!\!
\begin{array}{c} \nu-\varkappa\\ k' \end{array}\!\!\right] $.
\end{proof}

\begin{remark}
Changing $(l,k)$ to $(k,l)$ yields the {\it dual\/} $q$-Pascal graph
$\Gamma^*(q)$, which has the same set of vertices and edges as $\Gamma(q)$, but
different weights: the edge $(l,k)\to (l,k+1)$ has now weight $1$, and the edge
$(l,k)\to (l+1,k)$ has weight $q^k$. The sum of weights of paths in $\Gamma^*$
from $(0,0)$ to $(l,k)$ is again (\ref{dim}), which is related to  another
recursion for $q$-binomial coefficients, $d_{n,k}= d_{n-1,k-1}+ q^k d_{n-1,k}.$
\end{remark}

Consider the recursion
\begin{equation}\label{dual}
v_{n,k}= v_{n+1,k}+q^{n-k}v_{n+1,k+1}, ~~~~{\rm with ~~}v_{0,0}=1,
\end{equation}
which is dual to (\ref{rec1}), and  denote by $\V$ the set of nonnegative
solutions to \eqref{dual}.

\begin{proposition}\label{prop2B}
Formula
$$
{\mathbb P}\{S_n=(n-k,k)\}= d_{n,k}v_{n,k}, \qquad (n-k,k)\in \Gamma
$$
establishes
a bijective correspondence $\mathbb P\leftrightarrow v$ between the
probability laws of central Markov chains $S=(S_n)$ on $\Gamma(q)$ and
solutions $v\in\V$ to recursion \eqref{dual}.
\end{proposition}

\begin{proof}
Let $S$ be a central Markov chain on $\Ga$ with probability law $\prob$.
Observe that the property in Definition \ref{centr} means precisely that the
one-step {\it backward\/} transition probabilities (that is, transition
probabilities in the inverse time) are of the standard form
\begin{gather}
\prob\{S_{n-1}=(n-1,k)\mid S_n=(n,k)\}=\frac{d_{n-1,k}}{d_{n,k}}
=\frac{[n-k]}{[n]}\label{pro1}\\
\prob\{S_{n-1}=(n-1,k-1)\mid S_n=(n,k)\}
=\frac{d_{n-1,k-1}q^{n-k}}{d_{n,k}}=q^{n-k}\frac{[k]}{[n]}\label{pro2}
\end{gather}
for every such $S$.

Introduce the notation
\begin{equation}\label{tvnk}
\tv_{n,k}:={\mathbb P}\{S_n=(n-k,k)\}, \qquad (n-k,k)\in \Gamma.
\end{equation}
Consistency of the distributions of $S_n$'s amounts to the rule of total
probability
\begin{multline}\label{backw}
\tv_{n,k}=\prob\{S_{n}=(n,k)\mid
S_{n+1}=(n+1,k)\}\tv_{n+1,k}\\+\prob\{S_{n}=(n,k)\mid
S_{n+1}=(n+1,k+1)\}\tv_{n+1,k+1}.
\end{multline}
Rewriting \eqref{backw}, using \eqref{pro1} and \eqref{pro2},  and setting
\begin{equation}\label{rel}
v_{n,k}=d_{n,k}^{-1}\tv_{n,k}
\end{equation}
we get \eqref{dual}, which means that $v\in\V$. Thus, we have constructed the
correspondence $\prob\mapsto v$.

Conversely, start with a  solution $v\in\V$ and pass to $\tv=(\tv_{n,k})$
according to \eqref{rel}. For each $n$ consider the measure on $\Gamma_n$ with
weights $\tv_{n,0},\dots,\tv_{n,n}$. Since the weight of the root is $1$, it
follows from \eqref{dual} by induction in $n$ that these are probability
measures. Again by \eqref{dual}, the marginal measures are consistent with the
backward transition probabilities, hence determine the probability law of a
central Markov chain on $\Gamma(q)$. Thus, we get the inverse correspondence
$v\mapsto\prob$.
\end{proof}

By virtue of Propositions \ref{prop2A} and \ref{prop2B}, the law of
$q$-exchangeable infinite binary sequence is determined by some $v\in\V$,  with
the entries $v_{n,k}$ having  the same meaning as in \eqref{qvnk}. In the
sequel this law will be sometimes denoted $\prob_v$.

\section{The boundary problem}\label{3}

The set $\V$ is a Choquet simplex, meaning a convex set which is compact in the
product topology of the space of functions on $\Gamma$ and has the property of
uniqueness of the barycentric decomposition of each $v\in\V$ over the set of
extreme elements of $\V$ (see, e.~g., \cite[Proposition 10.21]{Goodearl}).

The {\it boundary problem\/} for the $q$-Pascal graph amounts to describing
extreme nonnegative solutions to the recursion (\ref{dual}). Each extreme
solution $v\in\V$ corresponds to ergodic process $(S_n)$ for which the tail
sigma-algebra is trivial.
In this context, the set of extremes is also known as
{\it the minimal boundary}.

With each array  $v\in\V$, $v=(v_{n,k})$, it is convenient to associate another
array $\tilde{v}=(\tilde{v}_{n,k})$ related to $v$ via (\ref{rel}). Clearly,
the mapping $v\leftrightarrow\tilde{v}$ is an isomorphism of two Choquet
simplexes $\V$ and $\tV=\{\tilde{v}\}$. Recall that the meaning of the
quantities $\tv_{n,k}$ is explained in \eqref{tvnk}.

A common approach to the boundary problem calls for identifying a {\rm
possibly} larger {\it Martin boundary} (see \cite{KOO}, \cite{Gibbs},
\cite{Euler} for applications of the method). To this end, we need to consider
multistep backward transition probabilities, which by Lemma \ref{dim}  are
given by a $q$-analogue of the hypergeometric distribution
\begin{multline}\label{multist}
\tilde v_{n,k}(\nu,\varkappa):={\mathbb P}\{S_n=(n-k,k)\,|\,S_\nu=(\nu-\varkappa,\varkappa)\}\\
=q^{(\varkappa-k)(n-k)}\, { \left[\!\!
\begin{array}{c} \nu-n\\ \varkappa-k \end{array}\!\!\right] \left[\!\!
\begin{array}{c} n\\ k \end{array}\!\!\right] \bigg/ \left[\!\!
\begin{array}{c} \nu\\ \varkappa \end{array}\!\!\right]},~~~~k=0,\ldots,n,
\end{multline}
and to examine the limiting regimes for $\varkappa=\varkappa(\nu)$ as
$\nu\to\infty$,  under which the probabilities (\ref{multist}) converge for all
{\rm fixed} $(n-k,k)\in\Gamma$. If the limits exist, the limiting array
$$
\quad \tilde v_{n,k}:=\lim_{(\nu,\varkappa)} \tilde v_{n,k}(\nu,\varkappa)
$$
belongs necessarily  to $\tV$.

Suppose $0<q<1$ and introduce polynomials
\begin{equation}\label{defPhi}
\Phi_{n,k}(x):=q^{-k(n-k)}x^{n-k}(x,q^{-1})_k, \qquad
\Phit_{n,k}=d_{n,k}\Phi_{n,k}\,, \qquad 0\le k\le n.
\end{equation}
Obviously, the degree of $\Phi_{n,k}$ is $n$; we will consider the polynomial
as a function on $\Delta_q$. Observe also that $\Phi_{n,k}(x)$ vanishes at
points $x=q^\ka$ with $\ka<k$, because of vanishing of $(x,q^{-1})_k$.

\begin{lemma}\label{MB}
Suppose $0<q<1$, and let in {\rm (\ref{multist})}  the indices $n$ and $k$
remain fixed, while $\nu\to\infty$ and $\varkappa=\varkappa(\nu)$ varies in
some way with $\nu$. Then the limit of \eqref{multist} is $\Phit_{n,k}(q^\ka)$
if $\ka$ is constant for large enough $\nu$. If $\ka\to\infty$ then the limit
is $\Phit_{n,k}(0)=\delta_{n,k}$.
\end{lemma}
\begin{proof}
 Assume first $\varkappa\to\infty$ and show that the limit of
(\ref{multist}) is $\delta_{nk}$. Since the quantities $\tilde
v_{n,k}(\nu,\varkappa)$, where $k=0,\dots,n$, form a probability distribution,
it suffices to check that the limit exists and is equal to 1 for $k=n$. In this
case the right--hand side of (\ref{multist}) becomes
$$
\prod_{i=1}^n \frac{[\varkappa-n+i]}{[\nu-n+i]}\,.
$$
Because $\lim_{m\to\infty}[m]=1/(1-q)$ for $q<1$, this indeed converges to $1$
provided that  $\varkappa\to\infty$.

Now suppose $\varkappa$ is fixed for all large enough $\nu$. The {\rm
right-hand side} of (\ref{multist}) is $0$ for $k>\varkappa$. For
$k\leq\varkappa$
 using $\lim_{m\to\infty}[m-j]!/[m]!=(1-q)^j$ we obtain
\begin{multline}
{ \left[\!\! \begin{array}{c} \nu-n\\ \varkappa-k \end{array}\!\!\right] \bigg/
\left[\!\! \begin{array}{c} \nu\\ \varkappa \end{array}\!\!\right]}=
\frac{[\nu-n]!}{[\nu]! }\,\,\frac{ [\nu-\varkappa]!}
{[\nu-\varkappa-(n-k)]!}\,\,\frac{[\varkappa]!}{[\varkappa-k]!}\\
\to \frac{(1-q)^k[\varkappa]!}{[\varkappa-k]!}=\Phit_{n,n}(q^\ka).
\end{multline}
\end{proof}

Part (i) of the next theorem  appeared in \cite[Chapter 1, Section 4, Corollary
6]{Kerov}. Kerov pointed out that the proof could be concluded from the
Kerov-Vershik `ring theorem' (see \cite[Section 8.7]{Zigzag}), but did not give
details.

For $\mu$ a measure, we shall write $\mu(x)$ instead of $\mu(\{x\})$, meaning
atomic mass at $x$.
\begin{theorem}\label{main}
Assume $0<q<1$.

{\rm(i)} The formulas
$$
\tv_{n,k}=\sum_{x\in\De_q}\Phit_{n,k}(x)\mu(x), \qquad
v_{n,k}=\sum_{x\in\De_q}\Phi_{n,k}(x)\mu(x)
$$
establish a linear homeomorphism  between the set $\tV$ (respectively,  $\V$)
and the set of all probability measures $\mu$ on $\De_q$.

{\rm(ii)} Given $\tv\in\tV$, the corresponding measure
$\mu$ is determined by
$$
\mu(q^\ka)=\lim_{\nu\to\infty}\tv_{\nu,\ka}, \qquad \ka=0,1,\dots;
\qquad \mu(0)=1-\sum_{\ka\in\{0,1,\ldots\}} \mu(q^\ka).
$$
\end{theorem}

\begin{proof}
As in  \cite{KOO}, the assertions  (i) and (ii) are consequences of the
following claims (a), (b) and (c).

(a) For each $\nu=0,1,2,\dots$, the vertex set $\Gamma_\nu$ is embedded into $\De_q$
via the map $(\nu,\ka)\mapsto q^\ka$. Observe that, as
$\nu\to\infty$, the image of $\Gamma_\nu$ in $\De_q$ expands and in
the limit exhausts the whole set $\De_q$, except point 0, which is a limit
point. In this sense, $\De_q$ is approximated by the sets
$\Gamma_\nu$ as $\nu\to\infty$.

(b) The multistep backward transition probabilities (\ref{multist}) converge to
$\Phit_{n,k}(q^\ka)$, for $0\leq \ka\leq \infty$,  in the regimes described by
Lemma \ref{MB}.

(c) The linear span of the functions $\Phit_{n,k}(x)$, $(n-k,k)\in\Gamma$, is the
space of all polynomials, so that it is dense in the Banach space $C(\De_q)$.
\end{proof}

Note that part (ii) of the theorem can be rephrased as follows: given
$\tv\in\tV$, consider the probability distribution on $\Gamma_n$ determined by
 $\tv_{n,\bullet}$ and take its pushforward under the embedding
$\Gamma_\nu\hookrightarrow\De_q$. The resulting probability measure
on $\De_q$ weakly converges to $\mu$ as $n\to\infty$.

\begin{corollary} For  $0<q<1$ we have:
\begin{itemize}
\item[\rm (i)] The extreme elements of\/ $\V$ are parameterised by the points
$x\in\De_q$ and have the form
\begin{equation}\label{Phi}
v_{n,k}=\Phi_{n,k}(x), \qquad 0\le k\le n.
\end{equation}
\item[\rm (ii)] The Martin boundary of the graph $\Ga(q)$ coincides with its
minimal boundary and can be identified with $\De_q\subset[0,1]$ via the
function $v\mapsto v_{1,0}$.
\end{itemize}
\end{corollary}

\begin{proof} All the claims are immediate. We only comment on the fact the
parameter $x\in\De_q$ is recovered as the value of $v_{1,0}$: this holds
because $\Phi_{1,0}(x)=x$.
\end{proof}
Letting $q\to 1$ we have a phase transition: the discrete boundary $\Delta_q$
becomes  more and more dense and eventually fills the whole of  $[0,1]$ at
$q=1$.

As is seen from \eqref{defPhi}, the polynomial $\Phi_{n,k}(x)$ can be viewed as a
$q$-analogue of the polynomial $x^{n-k}(1-x)^k$, so that \eqref{Phi} is a
$q$-analogue of \eqref{Haus}.
 Keep in mind that $x=q^\ka$ is a counterpart of $1-p$,
the probability of $\epsi_1=0$. The following $q$-analogue of the Hausdorff
problem of moments emerges. Introduce a modified  difference operator
acting on sequences $u=(u_l)_{l=0,1,\dots}$ as
$$
(\de_q u)_l=q^{-l}(u_{l}-u_{l+1}), \qquad l=0,1,\dots\,.
$$

\begin{corollary}
Assume $0<q<1$. A real sequence $u=(u_l)_{l=0,1,\dots}$ with $u_0=1$ is a
moment sequence of a probability measure $\mu$ supported by $\De_q\subset[0,1]$
if and only if $u$ is `$q$-completely monotone' in the sense that for every
$k=0,1,\dots$ we have componentwise
$$
\underbrace{\de_q\circ\dots\circ\de_q}_k u\ge0, \qquad k=0,1,\dots.
$$
\end{corollary}

\begin{proof}
Using the notation $v_{l+k,k}=u^{(k)}_l$ as in \eqref{uv}, we see that the
recursion (\ref{dual}) is equivalent to $u^{(k)}=\de_q u^{(k-1)}$, cf.
\eqref{de}. Then we use the fact that $\Phi_{n,0}(x)=x^n$ and repeat in the
reverse order the argument of Section \ref{1}.
\end{proof}

\subsection*{\sf The case $q>1$.}

This case can be readily reduced to the case with parameter $0<\bar{q}<1$,
where $\bar{q}:=q^{-1}$.  It is convenient to adopt a more detailed notation
$[n]_q$ for the $q$-integers.

\begin{lemma}
For every $q>0,\, \bar{q}=q^{-1}$, the backward transition probabilities \eqref{pro1}, \eqref{pro2}
for the graph $\Gamma(q)$ and the dual graph $\Gamma^*(\bar{q})$ are the same.
\end{lemma}
\begin{proof}

Indeed, by virtue of \eqref{pro1}, \eqref{pro2}, this is reduced to the
equality
$$
\frac{[n-k]_q}{[n]_q}=\bar{q}^{k}  \frac{[n-k]_{\bar{q}}}{[n]_{\bar{q}}}\,.
$$
\end{proof}

The lemma implies that the boundary problem for $q>1$ can be treated by passing
to $q^{-1}<1$ and changing $(l,k)$ to $(k,l)$. In terms of the  binary encoding
of the path, this means switching 0's with 1's.

Kerov  \cite[Chapter 1, Section 2.2]{Kerov}  gives more examples of `similar'
graphs, which have different edge weights but the same backward transition
probabilities.

\section{Examples}\label{4}

\subsection*{\sf A $q$-analogue of the Bernoulli process.} Our first example is a
description of the extreme $q$-exchangeable infinite binary sequences.

With each infinite binary sequence we associate some {\it $T$-sequence}
$(T_0,T_1,T_2,\dots)$ of nonnegative integers, where $T_j$ is the length of
$j$th run of 0's. That is to say, $T_0$ is the number of 0's before the first
1, $T_1$ is the number of 0's between the first and second 1's, $T_2$ is the
number of 0's between the second and third 1's, and so on. Clearly, this is a
bijection, i.e.  a binary sequence can be recovered from its $T$-sequence as
$$
(\,\underbrace{0,\dots,0}_{T_0}\,, 1,\underbrace{0,\dots,0}_{T_1}\,,1,
\underbrace{0,\dots,0}_{T_2}\,,1,\dots).
$$

If $q=1$, then the Bernoulli process with parameter $p$ has a simple
description in terms of the associated random $T$-sequence: all $T_i$ are
independent and have the same geometric distribution with parameter $1-p$.

\begin{proposition}
Assume $0<q<1$. For $x\in\De_q$, let $v(x)=(v_{n,k}(x))$ be the extreme element
of $\V$ corresponding to $x$. Consider $q$-exchangeable infinite binary
sequence $\epsi=(\epsi_1,\epsi_2,\dots)$ under the probability law\/
$\prob_{v(x)}$ and let $(T_0, T_1,\dots)$ be the associated random
$T$-sequence.

{\rm(i)} If $x=q^\ka$ with $\ka=1,2,\dots$ then $T_0,\dots,T_{\ka-1}$ are
independent, $T_\ka\equiv\infty$, and $T_i$ has geometric distribution
with parameter $q^{\ka-i}$ for  $0\le i\le\ka-1$.

{\rm(ii)} If $x=1$ then $T_0\equiv\infty$, which means that  with probability
one $\epsi$ is the sequence $(0,0,\ldots)$ of only $0$'s.

{\rm(iii)} If $x=0$ then $T_0\equiv T_1\equiv \dots\equiv0$, which means that
with probability one $\epsi$ is the sequence $(1,1,\ldots)$ of only $1$'s.
\end{proposition}

\begin{proof}
Consider the central Markov chain $S=(S_n)$ corresponding to the extreme element
$v(q^\ka)$. Computing the forward transition probabilities, from (\ref{Phi})
and (\ref{pro1}), for $0\leq k\leq\ka$ we have
\begin{multline}
{\mathbb P}\{S_{n+1}=(n+1-k,k) \mid S_n=(n-k,k)\}\\
=\frac{(q^{n+1-k}-1)}{(q^{n}-1)}\, \frac{d_{n+1,k}\,\Phi_{n+1,k}(q^\varkappa)}
{d_{n,k}\,\Phi_{n,k}(q^\varkappa)} = q^{\varkappa-k}.
\end{multline}
This implies (i) and (ii). In the limit case $x=0$ corresponding to
$\ka\to +\infty$, the above probability equals 0, which entails (iii).
\end{proof}

The analogy with the Bernoulli process is evident
from the above description of the binary sequence $\epsi(q^\ka)$.
Moreover, the Bernoulli process
appears  as a limit. Indeed, fix $p\in (0,1)$ and suppose $\ka$
varies with $q$, as $q\uparrow1$, in such a way that
$$
\ka\sim \frac{-\log(1-p)}{1-q}.
$$
In this limiting regime, $q^{\ka-k}\to 1-p$ for every $k$, hence $(T_0,T_1,
\dots)$ weakly converges to an infinite sequence of i.i.d. geometric variables
with parameter $1-p$, and the random binary sequence $\epsi(q^\ka)$ converges
in distribution to the Bernoulli process with the frequency of 0's equal to
$1-p$.

\subsection*{\sf Another $q$-analogue of Bernoulli process.}

Following \cite{Kerov}, another $q$-analogue  of Bernoulli process is suggested
by the $q$-binomial formula (see \cite{Kac})
$$
(-\theta,q)_n= \sum_{k=0}^n q^{k(k-1)/2} \left[\!\! \begin{array}{c} n\\ k
\end{array}\!\!\right] \theta^k.
$$
For $\theta\in [0,\infty]$ we
define a probability law $\prob_{w^\theta}$ for $S=(S_n)$ by
setting
\begin{equation}\label{vtheta}
w_{n,k}^\theta: = \frac{\theta^k q^{k(k-1)/2}}{(-\theta,q)_n}\,,\quad
\prob_{w^\theta}\{S_n=(n-k,k)\}:= d_{n,k} w_{n,k}^\theta, \quad (n,k)\in
\Gamma.
\end{equation}
Checking (\ref{dual}) is immediate.
Computing forward transition probabilities,
$$
{\mathbb P}_{w^\theta}\{S_{n+1}=(n+1-k,k)\,\giv\, S_n=(n-k,k)\}=1/(1+\theta q^n),
$$
shows that under $\prob_{w^\theta}$ the process $S_n=(n-K_n,K_n)$ has
independent inhomogeneous increments, with probability  $\theta
q^{n-1}/(1+\theta q^{n-1})$  for increment $K_n-K_{n-1}=1$. For $q=1$ we are
back to the ergodic Bernoulli process, but for $0<q<1$ the process is not
extreme. To obtain the barycentric decomposition of $w^\theta$ over extremes,
$$w^\theta=\sum_{0\leq\varkappa\leq\infty} v^\ka \mu(q^\ka),$$
we can apply Theorem \ref{main}(ii) to compute from (\ref{vtheta})
$$
\mu(q^\ka)=\lim_{n\to\infty}\prob_{w^\theta}\{S_n=(n-\ka,\ka)\}
=\frac{1}{(-\theta,q)_\infty}\,
\frac{q^{\ka(\ka-1)/2}\theta^\ka}{(1-q)^\ka[\ka]!}.
$$
This measure $\mu$ may be viewed as a $q$-analogue of the Poisson distribution.

\subsection*{\sf A $q$-analogue of P\'olya's urn process.}

The conventional P\'olya's urn process is described in \cite[Section
7.4]{Feller}. Here we provide its natural deformation.

Fix $a,\,b>0$ and $0<q<1$. Consider the Markov chain $(S_n)$ on $\Ga$ with the
forward transition probabilities from $(n-k,k)$ to $(n+1-k,k)$ and from
$(n-k,k)$ to $(n-k,k+1)$ given by
$$
\frac{[b+n-k]}{[a+b+n]} \quad {\rm and} \quad \frac{[a+k]}{[a+b+n]}\,q^{n-k+b},
$$
respectively. Then the distribution at time $n$ is
\begin{multline}\label{PU}
{\mathbb P}\{S_n=(n-k,k)\}= \left[\!\! \begin{array}{c} n\\ k
\end{array}\!\!\right]
q^{bk}\\
\times\frac{[a][a+1]\cdots[a+k-1][b][b+1]\cdots[b+n-k-1]}
{[a+b][a+b+1]\cdots[a+b+n-1]}.
\end{multline}
Checking consistency (\ref{dual}) is easy. The conventional P\'olya's urn
process appears in the limit $q\to 1$. The corresponding probability measure
$\mu$ is computable  from Theorem \ref{main}(ii) as
$$
\lim_{n\to\infty}\prob\{S_n=(n-\ka,\ka)\}
$$
For $a=1$, the limit distribution of the coordinate $\ka$ is geometric with
parameter $1-q^b$. For general $a,b$ we obtain a measure on $\Delta_q$
$$
\mu(q^{\ka})=\frac{(q^a,q)_\ka(q^b,q)_\infty}{(q,q)_\ka
(q^{a+b};q)_\infty}\,q^{\ka b}, \qquad q^\ka\in\De_q,
$$
which may be viewed as a $q$-analogue of the beta distribution on $[0,1]$.

\section{Grassmannians over a finite field}\label{5}

For $q$ a power of a prime number, let $\F_q$ be the Galois field with $q$
elements. Define $V_n$ to be the $n$-dimensional space of sequences
$(\xi_1,\xi_2,\ldots)$ with entries from $\F_q$, which satisfy $\xi_i=0$ for
$i>n$. The spaces $\{0\}=V_0\subset V_1\subset V_2\subset\dots$ comprise  a
complete flag, and the union  $V_\infty:=\cup_{n\geq 0}V_n$ is a countable,
infinite-dimensional space over $\F_q$.

By the {\it Grassmannian\/} $\Gr(V_\infty)$ we mean the set of all vector
subspaces $X\subseteq V_\infty$. Likewise, for $n\geq 0$ let $\Gr(V_n)$ be the
set of all vector subspaces in $V_n$, with  $\Gr(V_0)$ being a singleton.
Consider the projection $\pi_{n+1,n}:\Gr(V_{n+1})\to\Gr(V_n)$ which sends a
subspace of $V_{n+1}$ to its intersection with $V_n$.

\begin{lemma}\label{lemma1}
There is a canonical  bijection
 $X\leftrightarrow (X_n)$ between the Grassmannian $\Gr(V_\infty)$ and
the set of sequences $(X_n\in\Gr(V_n),~ n\geq 0)$ satisfying the consistency
condition $X_n=\pi_{n+1,n}(X_{n+1})$ for each $n$.
\end{lemma}

\begin{proof}
Indeed, the mapping $X\mapsto(X_n)$ is given by setting $X_n=X\cap V_n$ for
each $n$, while the  mapping $(X_n)\mapsto X$ is defined by $X=\cup X_n$.
\end{proof}

The lemma shows that $\Gr(V_\infty)$ can be identified with a projective limit
of the finite sets $\Gr(V_n)$, the projections being the maps $\pi_{n+1,n}$.
Using this identification we endow $\Gr(V_\infty)$ with the corresponding
topology, in which $\Gr(V_\infty)$ becomes a totally disconnected compact
space. For $X\in\Gr(V_\infty)$, a fundamental system of its neighborhoods is
comprised of the sets of the form $\{X'\in\Gr(V_\infty):~ X'_n=X_n\}$, where
$n=1,2,\dots$\,.

Let $\G_n=GL(n,\F_q)$ be the group of invertible linear
transformations of the space $V_n$, realised as the group of transformations
of $V_\infty$ which may only change the first $n$ coordinates.
We have then
$\{e\}=\G_0\subset\G_1\subset\G_2\subset\dots$ and we define
$\G_\infty:=\cup\G_n$.
The countable group $\G_\infty$  consists of
infinite invertible matrices $(g_{ij})$,
such that $g_{ij}=\delta_{ij}$ for large enough  $i+j$.
The group  $\G_\infty$ acts on $V_\infty$
hence also acts on $\Gr(V_\infty)$.

A probability distribution on $\Gr(V_\infty)$ defines a random subspace of
$V_\infty$. We look at random subspaces of $V_\infty$ whose distribution is
invariant under the action of $\G_\infty$. Observe that the action of $\G_n$
splits $\Gr(V_n)$ into orbits
$$
G(n,k)=\{X\in \Gr(V_n),\, \dim X=k\}, ~~0\leq k\leq n,
$$
where $\#G(n,k)=d_{n,k}$ is the number of $k$-dimensional subspaces of $V_n$.
Therefore, a probability  distribution on $\Gr(V_\infty)$ is
$\G_\infty$-invariant if and only if the conditional distribution on each
$G(n,k)$ is uniform.

It must be clear that this setting of `$q$-exchangeability' of linear spaces is
analogous to the framework of de Finetti's theorem: exchangeability of a random
binary sequence means that the conditional measure is uniform on sequences of
length $n$ with $k$ 1's. See \cite{Aldous}, \cite{diaconis} for more on
symmetries and sufficiency.

\begin{lemma} Formula
$$
\tilde{v}_{n,k} = P\{X\in\Gr(V_\infty): X\cap V_n\in G(n,k)\}, \quad
(n,k)\in\Gamma
$$
establishes a linear homeomorphism between $\tV$ and $\G_\infty$-invariant
probability measures on the Grassmannian $\Gr(V_\infty)$.
\end{lemma}

\begin{proof} We first spell out more carefully the remark before the lemma.
Consider projections
$$
\pi_{\infty,n}: \Gr(V_\infty)\to \Gr(V_n), \qquad X\mapsto X\cap V_n, \quad
X\in\Gr(V_\infty), \quad n=1,2,\dots\, .
$$
If $P$ is a Borel probability measure on the space $\Gr(V_\infty)$,
then, for any $n$, the pushforward $P_n:=\pi_{\infty,n}(P)$ is a
probability measure on $\Gr(V_n)$, and the measures $P_n$ are
consistent with respect to the projections $\pi_{n+1,n}$, that is,
$$
P_n=\pi_{n+1,n}(P_{n+1}), \qquad n=0,1,2,\dots\,.
$$
Conversely, if a sequence $(P_n)$ of probability measures is consistent, then
it determines a probability measure $P$ on $\Gr(V_\infty)$. Moreover, $P$ is
$\G_\infty$-invariant if and only if each $P_n$ is $\G_n$-invariant. Next,
observe that if $P_n$ is a $\G_n$-invariant probability measure, then it
assigns the same weight to each $k$-dimensional space $X_n\in G(n,k)$; let us
denote this weight by $v_{n,k}$.

Fix $X_n\in G(n,k)$. We claim that there are precisely $q^{n-k}+1$ subspaces
$X_{n+1}\in\Gr(V_{n+1})$ such that $X_{n+1}\cap V_n=X_n$: one subspace from
$G(n+1,k)$  and  $q^{n-k}$ subspaces from $G(n+1,k+1)$. Indeed, $\dim X_{n+1}$
equals either $k$ or $k+1$. In the former case $X_{n+1}=X_n$, while in the
latter case $X_{n+1}$ is spanned by $X_n$ and a nonzero vector from
$V_{n+1}\setminus V_{n}$. Such a vector is defined uniquely up to a scalar
multiple and addition of an arbitrary vector from $X_n$. Therefore, the number
of options is equal to the number of lines in $V_{n+1}/X_n$ not contained in
$V_n/X_n$, which equals
$$
\frac{q^{n+1-k}-1}{q-1}\, -\, \frac{q^{n-k}-1}{q-1}=q^{n-k}.
$$

Now, let $P$ be a $\G_\infty$-invariant probability measure on $\Gr(V_\infty)$,
with projections $(P_n)$ specified by the corresponding array of weights
$v=(v_{n,k})$. Then the relations $P_n=\pi_{n+1,n}(P_{n+1})$ together with the
dimension computation imply that $v$ satisfies (\ref{dual}).

Conversely, given $v\in\V$, we can construct a sequence  $(P_n)$ of measures
such that $P_n$ lives on $\Gr(V_n)$, is invariant under $\G_n$ and agrees with
$P_{n+1}$ under $\pi_{n+1,n}$. Since $P_0$, which lives on a singleton, is
obviously a probability measure, we obtain by induction that all $P_n$ are
probability measures. Taking their projective limit we get a
$\G_\infty$-invariant probability measure $P$ on $\Gr(V_\infty)$.
\end{proof}

Rephrasing Theorem \ref{main} we have from the lemma

\begin{corollary} The ergodic $\G_\infty$-invariant probability measures on
$\Gr(V_\infty)$ are parameterised by $\ka\in \{0,1,\ldots,\infty\}$. For
$\ka=0$ the measure is the Dirac mass at $V_\infty$, for $\ka=\infty$ it is the
Dirac mass at $V_0$, and for $0<\ka<\infty$ the measure is supported by the set
of subspaces of\/ $V_\infty$ of codimension $\ka$.
\end{corollary}

The following random algorithm describes explicitly the dynamics of the growing
space $X_n\in \Gr(V_n)$ as $n$ varies, under the ergodic measure with parameter
$\ka$. Recall the notation $\bar{q}=q^{-1}$. Start with $X_0=V_0$. With
probability $\bar{q}^\ka$ choose $X_1=V_1$, and with probability
$1-\bar{q}^\ka$ choose $X_1=X_0$. Suppose $X_n\subseteq V_n$ has been
constructed and has dimension $n-k$ with $k\leq\ka$. Then let $X_{n+1}=X_n$
with probability $1-\bar{q}^{\ka-k}$, and with probability $\bar{q}^{\ka-k}$
choose uniformly at random a nonzero vector $\xi\in V_{n+1}\setminus V_n$ and
let $X_{n+1}$  be the linear span of $X_n$ and $\xi$.

\subsection*{\sf Duality.}

We finish with  a dual version of our construction. Let $V^\infty$ denote the
set of all sequences $\eta=(\eta_1,\eta_2,\dots)$ with entries from $\F_q$.
This is again a vector space over $\F_q$, strictly larger than $V_\infty$ since
we do not require $\eta$ to have finitely many nonzero entries. That is to say,
$V^\infty$ is just the infinite product space $(\F_q)^\infty$, which we endow
with the product topology. Let $\Gr(V^\infty)$ denote the set of all {\it
closed\/} subspaces $Y\subseteq V^\infty$. A dual version of Lemma \ref{lemma1}
says that such subspaces $Y$ are in a bijective correspondence with the
sequences $(Y_n\in\Gr(V_n), n\geq 0)$ such that $Y_n=\pi'_{n+1,n}(Y_{n+1})$,
where $\pi'_{n+1,n}$ is induced by the projection map $V_{n+1}\to V_n$ which
sets the $(n+1)$th coordinate of a vector $\xi\in V_{n+1}$ equal to $0$. The
branching of $G(n,k)$'s under these projections corresponds to the dual
$q$-Pascal graph.

\begin{lemma}
The operation of passing to the orthogonal complement with respect to the
bilinear form
$$
\langle\xi,\eta\rangle:=\sum_{i=1}^\infty \xi_i\eta_i, \qquad \xi\in
V_\infty, \quad \eta\in V^\infty,
$$
is a  bijection
$\Gr(V_\infty)\leftrightarrow\Gr(V^\infty)$.
\end{lemma}

\begin{proof} First of all, note that the bilinear form
is well defined, because the coordinates $\xi_i$ of $\xi\in V_\infty$
vanish for $i$ large
enough. This form determines a bilinear pairing $V_\infty\times
V^\infty\to\F_q$. We claim that it brings the spaces $V_\infty$
and $V^\infty$ into duality, where $V^\infty$ is viewed as a
vector space with nontrivial topology, and  the topology on $V_\infty$ is
discrete.

Indeed, it is evident that the pairing is nondegenerate and that any linear
functional on $V_\infty$ is given by a vector of $V^\infty$. A minor reflection
also shows that, conversely, any {\it continuous\/} linear functional on
$V^\infty$ is given by a vector from $V_\infty$. Thus, the  spaces $V_\infty$
and $V^\infty$ are indeed dual to one another. They are also dual  as
commutative locally compact topological groups: one is discrete and the other
is compact.

Using the duality, it is readily checked that if $X$ is an arbitrary subspace
in $V_\infty$, then its orthogonal complement $X^\bot$ is a closed subspace in
$V^\infty$, whose orthogonal complement $(X^\bot)^\bot$ coincides with $X$.
Likewise, starting with a closed subspace $Y\subseteq V^\infty$, we have
$Y^\bot\subseteq V_\infty$ and $(Y^\bot)^\bot=Y$. Thus,  the operation of
taking the orthogonal complement is  a bijection.
\end{proof}

The group $\G_\infty$ acts on both  $V_\infty$ and $V^\infty$ and preserves the
pairing between these vector spaces. Under the identification
$\Gr(V^\infty)=\Gr(V_\infty)$, the group $\G_\infty$ acts by homeomorphisms on
this compact space. In the dual picture, the ergodic measures with $\ka<\infty$
live on the set of $\ka$-dimensional subspaces of $V^\infty$ . The case
$\ka=\infty$ corresponds then to the zero subspace in $V_\infty$ (or the full
space $V^\infty$). There is a simple explanation why we have to fix codimension
in the $V_\infty$-picture and dimension in the $V^\infty$-picture, and not vice
versa. Namely, the subspaces in $V_\infty$ of fixed nonzero finite dimension
form a countable set, which is a single $\G_\infty$-orbit, and such a
$\G_\infty$-space cannot carry a finite invariant measure.

\end{document}